\documentclass[10 pt]{article}
\usepackage{amssymb,amsmath,amsfonts}
\usepackage{graphics}
\usepackage[english]{babel}
\usepackage{graphicx}
\usepackage{float}
\usepackage{sidecap}
\numberwithin{equation}{section}

\newenvironment{proof}{\noindent {\em {Proof}}.}{$\square$
\medskip}
\newtheorem{definition}{Definition}
\newtheorem{corollary}{Corollary}
\newtheorem{ex}{Example}
\newtheorem{theorem}{Theorem}

\newtheorem{remark}{Remark}

\begin{document}
\title{On the eigenvalues of some signed graphs}
\author{M. Souri$^a$,~F. Heydari$^{a,}$\footnote{Corresponding author.\newline \hspace*{4mm} E-mail addresses: mona.souri@kiau.ac.ir (M. Souri), f-heydari@kiau.ac.ir, \newline faride.heydari@gmail.com (F. Heydari), maghasedi@kiau.ac.ir (M. Maghasedi).}, M. Maghasedi$^a$}
\date{ }

\maketitle \vspace{-9mm}

\begin{center}
\date {}{ \small \textit{$^a$Department of Mathematics, Karaj Branch, Islamic Azad University, Karaj, Iran}}
\end{center}

\begin{abstract}
Let $G$ be a simple graph and $A(G)$ be the adjacency matrix of $G$. The matrix $S(G) = J -I -2A(G)$ is called the Seidel matrix of $G$, where $I$ is an identity matrix and $J$ is a square matrix all of whose entries are equal to 1. Clearly, if $G$ is a graph of order $n$ with no isolated vertex, then the Seidel matrix of $G$ is also the adjacency matrix of a signed complete graph $K_n$ whose negative edges induce $G$. In this paper, we study the Seidel eigenvalues of the complete multipartite graph $K_{n_1,\ldots,n_k}$ and investigate its Seidel characteristic polynomial. We show that if there
are at least three parts of size $n_i$, for some $i=1,\ldots,k$,
then $K_{n_1,\ldots,n_k}$ is determined, up to switching, by its Seidel spectrum.
\end{abstract}

~~~\noindent\textbf{Keywords:} signed graph, eigenvalues, complete multipartite graphs, \newline \hspace*{9mm}Seidel matrix, $S$-determined.\\

~~~\noindent\textbf{MSC(2010):} 05C22; 05C31; 05C50.

\section{Introduction}
Let $G$ be a simple graph with the vertex set $V(G)$ and the edge set $E(G)$. The \textit {order} of $G$ is defined $\vert V(G)\vert $. A \textit{signed graph} $\Gamma$ is an ordered pair $(G,\sigma)$,
where $\sigma : E(G) \longrightarrow \lbrace -1,+1\rbrace$ is a sign function. If all edges of a signed graph are positive (resp., negative),
then we denote it by $(G,+)$ (resp., $(G,-)$). The signed graphs have been studied by many authors, for instance see \cite{Akbari1, Akbari2, Belardo}.

Let $A(G)=(a_{ij})$ be the adjacency matrix of $G$, and let $S(G) = J -I -2A(G)$ be its \textit{Seidel matrix}, where $I$ is an identity matrix and $J$ is a square matrix all of whose entries are equal to 1. The adjacency matrix of a signed graph $\Gamma=(G,\sigma)$ is a square matrix $A(\Gamma)=(a_{ij}^\sigma)$, where $a_{ij}^\sigma=\sigma (v_iv_j)a_{ij} $. As usual, $K_n$ and $\overline{K_n}$ will denote
the complete graph and the null graph of order $n$, respectively. Clearly, if $G$ is a graph of order $n$ and $\Gamma=(K_{n},G^-)$ denotes a signed complete graph $K_n$ in which the edges of $G$ are negative and the edges not in $G$ are positive, then $A(\Gamma)=S(G)$. The polynomial $S_G(\lambda)=\det(\lambda I-S(G))$ is called the {\it Seidel characteristic polynomial} of $G$ and its roots often referred to as the {\it Seidel eigenvalues} or {\it Seidel spectrum} of $G$.

When dealing with Seidel matrices, the concept of {\it Seidel switching} has an important role.
Let $U\subset V(G)$. A Seidel switching with respect to $U$ transforms $G$ to a graph $H$ by deleting the edges between $U$ and $V(G)\setminus U$ and adding all edges between them that were not present in $G$. It is easy to see that two Seidel matrices $S(G)$ and $S(H)$ are similar and so have the same spectrum. Seidel switching is an equivalence relation and two graphs $G$ and $H$ are called \textit{switching equivalent}.
We say that a graph $G$ is {\it Seidel determined, up to switching} (or {\it $S$-determined}) if the only graphs with same Seidel spectrum are switching equivalent to a graph isomorphic to $G$.

The complete multipartite graph $K_{n_1,\ldots,n_k}$ is a graph of order $n=\sum_{i=1}^{k}n_i$. The set of vertices is partitioned into
parts $V_1,\ldots,V_k$ such that $|V_i|=n_i$, for $i=1,\ldots,k$ and an edge joins two vertices if and only if they belong to different parts.
In \cite{berman}, the authors proved that
any graph which has the same Seidel spectrum as a complete $k$-partite graph is switching equivalent to a complete $k$-partite graph, and if the different partition sets sizes are $n_1,\ldots,n_s$, and there are at least three parts of each size $n_i,$ $i=1,\ldots,s$, then it is determined, up to switching, by its Seidel spectrum.
In this paper, we investigate the Seidel characteristic polynomial of the complete multipartite graph $K_{n_1,\ldots,n_k}$. We show that if two complete $k$-partite graphs $K_{n_1,\ldots,n_k}$ and $K_{q_1,\ldots,q_k}$ are Seidel cospectral and $n_i=q_j$ for some $i,j$, $1\leq i,j\leq k$, then
there is a permutation $\pi$ such that $n_i=q_{\pi(i)}$ for $i=1,\ldots,k$. As a corollary, we prove that if there
are at least three parts of size $n_i$, for some $i=1,\ldots,k$,
then $K_{n_1,\ldots,n_k}$ is $S$-determined. Among other results, we find an upper bound for the least Seidel eigenvalue of $K_{n_1,\ldots,n_k}$.

\section{The Seidel characteristic polynomial}

Let $G$ be the complete multipartite graph $K_{n_1,\ldots,n_k}$ of order $n$. The Seidel characteristic polynomial of $G$ was found in \cite{wang}. We include here a different proof by considering the Seidel matrix of $G$ as the adjacency matrix of a signed complete graph $K_n$ whose negative edges induce $G$. First we need the following definition.

\begin{definition}
{\rm(\cite[p.24]{Haemers})}
{\rm
Suppose $A$ is a symmetric real matrix whose rows and columns are indexed by
$X=\{1,\ldots,n\}$. Let $\{X_1,\ldots,X_k\}$ be a partition of $X$ and $A$ be
partitioned according to $\{X_1,\ldots,X_k\}$, that is,\\
\begin{center}
$\begin{bmatrix}
A_{1,1}&\ldots&A_{1,k}\\
\vdots&&\vdots  \\
A_{k,1}&\ldots &A_{k,k}\\
\end{bmatrix},$
\end{center}
where $A_{i,j}$ denote the submatrix of $A$ formed by rows in $X_i$ and the
columns in $X_j$. Let $b_{i,j}$ denote the average row sum of $A_{i,j}$. Then the matrix
$B=(b_{i,j})$ is called the {\it quotient matrix} of $A$ with respect to the given partition. If the row sum of each $A_{i,j}$ is constant, then the partition is called {\it equitable}.
}
\end{definition}

So we have the next theorem.

\begin{theorem}\label{seidel}
Let $G$ be the complete multipartite graph $K_{n_1,\ldots,n_k}$ of order $n$. Then
$$S_G(\lambda)=(\lambda +1)^{n-k} \Big(1+\sum_{i=1}^{k}\frac{n_i}{\lambda+1-2n_i}\Big)\prod_{i=1}^{k}(\lambda+1-2n_i).$$
\end{theorem}

\begin{proof}
{
Let $V_1,\ldots,V_k$ be the parts of $V(K_{n_1,\ldots,n_k})$ and $|V_i|=n_i$ for $i=1,\ldots,k$. The partition $\sqcap=\{V_1,\ldots,V_k\}$ is an equitable partition. By induction on $k$, we prove that the characteristic polynomial of the quotient matrix $B$ of $S(G)$ with respect to $\sqcap$ is as follows:
$$\Big(1+\sum_{i=1}^{k}\frac{n_i}{\lambda+1-2n_i}\Big)\prod_{i=1}^{k}(\lambda+1-2n_i).$$
One can easily see that the assertion is true for $k=2$. Let $\varphi_B(\lambda)=\det(\lambda I-B)$. So we have
$$\varphi_B(\lambda)=\det \begin{bmatrix}
\lambda +1-n_1 &n_2&\ldots&n_{k-1}& n_k\\
n_1 &\lambda +1-n_2 &\ldots&n_{k-1}&n_k  \\
\vdots &\vdots&\ddots &\vdots&\vdots \\
n_1 &n_2&\ldots&\ddots &n_k\\
n_1 &n_2&\ldots&n_{k-1} & \lambda +1-n_k
\end{bmatrix}.$$

If we subtract the last row from the other rows, then
\fontsize{9pt}{9pt}
$$\varphi_B(\lambda)=\det \begin{bmatrix}
\lambda +1-2n_1 &0&\ldots&\ldots&0& -\lambda -1+2n_k\\
0 &\lambda +1-2n_2 & & &&-\lambda -1+2n_k  \\
\vdots &&\ddots& &\bold{0}&\vdots \\
\vdots &\bold{0}&&\ddots &&\vdots \\
0 && & &\lambda +1-2n_{k-1} &-\lambda -1+2n_k\\
n_1 &n_2&\ldots&\ldots&n_{k-1} & \lambda +1-n_k
\end{bmatrix}.$$
\normalsize
Thus
\fontsize{9pt}{9pt}
$$\varphi_B(\lambda)=(\lambda +1-2n_1)\det \begin{bmatrix}
 \lambda +1-2n_2 & & &&-\lambda -1+2n_k  \\
&\ddots& &\bold{0}&\vdots \\
\bold{0}&&\ddots &&\vdots \\
& & &\lambda +1-2n_{k-1} &-\lambda -1+2n_k\\
n_2&\ldots&\ldots&n_{k-1} & \lambda +1-n_k
\end{bmatrix}$$
\vspace{10mm}
$\ \ \ \ \ \ \ \ \ \ \ \ \ \ \ \ \ \ +(-1)^{k+1}n_1\det \begin{bmatrix}
0&\ldots&\ldots&0& -\lambda -1+2n_k\\
\lambda +1-2n_2 & & &&-\lambda -1+2n_k  \\
&\ddots& &\bold{0}&\vdots \\
\bold{0}&&\ddots &&\vdots \\
& & &\lambda +1-2n_{k-1} &-\lambda -1+2n_k\\
\end{bmatrix}.$\\

By induction hypothesis, we have
$$\varphi_B(\lambda)=\Big(1+\sum_{i=2}^{k}\frac{n_i}{\lambda+1-2n_i}\Big)\prod_{i=1}^{k}(\lambda+1-2n_i)+n_1\prod_{i=2}^{k}(\lambda+1-2n_i),$$
and hence the result holds. On the other hand if $u$ and $v$ are in the same part, the vector $e_u-e_v$ whose coordinates on $u, v$ and elsewhere are
respectively 1, $-1$ and 0 is an eigenvector for the eigenvalue $-1$. This explains the factor $(\lambda +1)^{n-k}$. Clearly, two graphs $K_{n_1,n_2}$ and $\overline{K_n}$ are switching equivalent, where $n=n_1+n_2$. Thus two matrices $S(K_{n_1,n_2})$ and $S(\overline{K_n})=A((K_n,+))$ are similar. Therefore, $K_{n_1,n_2}$ has $-1$ as a Seidel eigenvalue with the multiplicity $n-1$. If $k>2$, then $-1$ is not an eigenvalue of $B$. Hence by \cite[Lemma 2.3.1]{Haemers}, the proof is complete.
}
\end{proof}

Suppose that the number of distinct integers of $n_1,\ldots,n_k$ is $s$. Without loss of generality, assume that $n_1,\ldots,n_s$
are distinct. Suppose that $r_i$ is the multiplicity of $n_i$, for $i=1,\ldots,s$.
The complete multipartite graph $K_{n_1,\ldots,n_k}=K_{n_1,\ldots,n_1,\ldots,n_s,\ldots,n_s}$ is also denoted by $K_{r_1.n_1,\ldots,r_s.n_s}$, where $k=\sum_{i=1}^{s}r_i$ and $n=\sum_{i=1}^{s}r_in_i$.
From the previous theorem, we have the next immediate result.

\begin{corollary}\label{cor1}
Let $G$ be the complete multipartite graph $K_{n_1,\ldots,n_k}=K_{r_1.n_1,\ldots,r_s.n_s}$ of order $n$. Then
$$S_G(\lambda)=(\lambda +1)^{n-k} \prod_{i=1}^{s}(\lambda+1-2n_i)^{r_i-1} \Big(\prod_{i=1}^{s}(\lambda+1-2n_i)+\sum_{i=1}^{s}r_in_i \prod_{j\neq i}(\lambda+1-2n_j)\Big).$$
\end{corollary}

\begin{remark}{\rm
By \cite[Lemma 3.2(1)]{berman}, the complete $k$-partite graph $K_{n_1,\ldots,n_k}$ has exactly $k-1$ (not necessarily distinct) positive Seidel eigenvalues. Moreover, if $n_1\geq n_2\geq \cdots \geq n_k$ and $\lambda_1\geq \cdots \geq \lambda_{k-1}$ are the positive Seidel eigenvalues of $K_{n_1,\ldots,n_k}$, then by \cite[Lemma 3.2(2)]{berman},
$$2n_k-1\leq \lambda_{k-1}\leq \cdots \leq2n_2-1\leq\lambda_1\leq 2n_1-1.$$
Also if $2<k<n=\sum_{i=1}^{k}n_i$, then
$K_{n_1,\ldots,n_k}$ has exactly two distinct negative Seidel eigenvalues $\lambda_n, -1$ ($\lambda_n<-1$) with multiplicity $1$ and $n-k$, respectively. (By the proof of Theorem \ref{seidel}, if $2<k<n=\sum_{i=1}^{k}n_i$, then the multiplicity of $-1$ is $n-k$.) Note that $S(K_n)=A((K_n,-))$ and $S(K_{n_1,n_2})$ is similar to $S(\overline{K_n})=A((K_n,+))$, where $n=n_1+n_2$.
}
\end{remark}

By the Rayleigh Principle \cite[p.31]{Haemers}, we have the following upper bound for $\lambda_n$.

\begin{theorem}
Let $\lambda_n$ be the least Seidel eigenvalue of the complete multipartite graph $K_{n_1,\ldots,n_k}$ of order $n$. Then
$$\lambda_n\leq \frac{n}{k}-1-\frac{2}{k}\sum_{1\leq i<j\leq k} \sqrt{n_in_j}.$$
\end{theorem}

\begin{proof}
{Let $V_1,\ldots,V_k$ be the parts of $V(K_{n_1,\ldots,n_k})$ and $|V_i|=n_i$ for $i=1,\ldots,k$. Suppose that
$B$ is the quotient matrix of $S(K_{n_1,\ldots,n_k})$ with respect to the equitable partition $\sqcap=\{V_1,\ldots,V_k\}$.
The matrix $B$ is similar to the symmetric matrix $\tilde{B}=D^{-1}BD$, where $D={\rm diag}(1/\sqrt{n_1},\ldots,1/\sqrt{n_k})$. Now by the Rayleigh quotient, the vector $[1,\ldots,1]$ gives the upper bound $\frac{n}{k}-1-\frac{2}{k}\sum_{1\leq i<j\leq k} \sqrt{n_in_j}$.
}
\end{proof}

Now, we compute the coefficients of the Seidel characteristic polynomial of the complete multipartite graph $K_{n_1,\ldots,n_k}$.

\begin{theorem} {\rm \cite[Theorem 1]{delor}}
The characteristic polynomial of the complete multipartite graph $K_{n_1,\ldots,n_k}$ is
$$\lambda^{n-k} \Big(\lambda^k-\sum_{m=2}^{k}(m-1)\sigma_m \lambda^{k-m}\Big),$$
where $n=\sum_{i=1}^{k}n_i$, and $\sigma_m=\sum_{1\leq j_1<\cdots<j_m\leq k} n_{j_1}\ldots n_{j_m}$ for $m=2,\ldots,k$.
\end{theorem}

\begin{theorem}\label{coef}
Let $G$ be the complete multipartite graph $K_{n_1,\ldots,n_k}$. Then
\begin{center}
$S_G(\lambda)=(\lambda +1)^{n-k} \Big(\lambda^k+(k-n)\lambda^{k-1}+(\binom{k}{2}-(k-1)n)\lambda^{k-2}$
\end{center}
$\ \ \ \ \ \ \ \ \ \ \ \ \ \ \ \ \ \ \ \ \ \ \ \ \ \ \ \ \ \ \ \ \ \ \ \ +\sum_{m=3}^{k}\sum_{i=0}^{m}(-1)^{i-1}2^{i-1}\binom{k-i}{m-i}(i-2)\sigma_i\lambda^{k-m}\Big),$

\vspace{3mm}
\noindent where $\sigma_0=1$, $n=\sigma_1=\sum_{i=1}^{k}n_i$, and $\sigma_i=\sum_{1\leq j_1<\cdots<j_i\leq k} n_{j_1}\ldots n_{j_i}$ for $i=2,\ldots,k$.
\end{theorem}

\begin{proof}
{By Theorem \ref{seidel}, we have
$$S_G(\lambda)=(\lambda +1)^{n-k} \Big(\prod_{i=1}^{k}(\lambda+1-2n_i)+\sum_{i=1}^{k}n_i\prod_{j\neq i}(\lambda+1-2n_j)\Big),$$
where $n=\sum_{i=1}^{k}n_i$. Clearly, $\prod_{i=1}^{m}(1-2x_i)=\sum_{i=0}^{m}(-1)^{i}2^i\tilde{\sigma}_i$, where $\tilde{\sigma}_0=1$ and $\tilde{\sigma_i}=\sum_{1\leq j_1<\cdots<j_i\leq m} x_{j_1}\ldots x_{j_i}$ for $i=1,\ldots,m$. Now, $\binom{k}{m}\binom{m}{i}=\binom{k}{i}\binom{k-i}{m-i}$ implies that the coefficient of $\lambda^{k-m}$ in $\prod_{i=1}^{k}(\lambda+1-2n_i)$ is
$$\sum_{i=0}^{m}(-1)^{i}2^i\binom{k-i}{m-i}\sigma_i,$$
where $\sigma_0=1$ and $\sigma_i=\sum_{1\leq j_1<\cdots<j_i\leq k} n_{j_1}\ldots n_{j_i}$ for $i=1,\ldots,k$.

Also, $x_1\prod_{i=2}^{m}(1-2x_i)=\sum_{i=1}^{m}(-1)^{i-1}2^{i-1}x_1\tilde{\sigma}_{i-1}$, where $\tilde{\sigma}_0=1$ and $\tilde{\sigma_i}=\sum_{2\leq j_1<\cdots<j_i\leq m} x_{j_1}\ldots x_{j_i}$ for $i=1,\ldots,m-1$. Hence $k\binom{k-1}{m-1}\binom{m-1}{i-1}=i\binom{k}{i}\binom{k-i}{m-i}$ yields that the coefficient of \underline{}$\lambda^{k-m}$ in $\sum_{i=1}^{k}n_i\prod_{j\neq i}(\lambda+1-2n_j)$ is
$$\sum_{i=1}^{m}(-1)^{i-1}2^{i-1}i\binom{k-i}{m-i}\sigma_i.$$
The proof is complete.
}
\end{proof}

In \cite{Akbari2} and \cite{LV}, the Seidel characteristic polynomials of two graphs $K_{n_1,n_2,n_3,n_4}$ and $K_{n_1,n_2,n_3}$ have been studied.
Here, we have the next example.

\begin{ex}{\rm
The Seidel characteristic polynomials of the complete multipartite graphs $K_{n_1,\ldots,n_k}$ for $k=3,4,5$ are respectively as follows:

$$(\lambda +1)^{n-3} \Big(\lambda^3+(3-n)\lambda^2+(3-2n)\lambda+1-n+4\sigma_3\Big),$$

$$(\lambda +1)^{n-4} \Big(\lambda^4+(4-n)\lambda^3+(6-3n)\lambda^2+(4-3n+4\sigma_3)\lambda+1-n+4\sigma_3-16\sigma_4\Big),$$
\\
$(\lambda + 1)^{n-5} \Big(\lambda^5+(5-n)\lambda^4+(10-4n)\lambda^3+(10-6n+4\sigma_3)\lambda^2$\\

$\ \ \ \ \ \ \ \ \ \ \ \ \ \ \ \ \ \ \ \  +(5-4n+8\sigma_3-16\sigma_4)\lambda+1-n+4\sigma_3-16\sigma_4+48\sigma_5\Big),$\\

\noindent where $n=\sum_{i=1}^{k}n_i$, and $\sigma_i=\sum_{1\leq j_1<\cdots<j_i\leq k} n_{j_1}\ldots n_{j_i}$ for $i=3,4,5$.
}
\end{ex}

\begin{theorem}\label{coef2}
Let $G$ be the complete multipartite graph $K_{n_1,\ldots,n_k}=K_{r_1.n_1,\ldots,r_s.n_s}$ of order $n$. Then\\

$S_G(\lambda)=(\lambda +1)^{n-k} \prod_{i=1}^{s}(\lambda+1-2n_i)^{r_i-1} \Big(\lambda^s+$\\

$\ \ \ \ \ \ \ \ \ \ \ \ \ \ \ \ \ \ \ \ \ \ \ +\sum_{m=1}^{s}\sum_{i=0}^{m}(-1)^{i-1}2^{i-1}\binom{s-i}{m-i}(\sum_{l=1}^{s}r_l\sigma_{li}-2\sigma_i)\lambda^{s-m}\Big),$

\vspace{3mm}
\noindent where $\sigma_0=1$, $\sigma_{l0}=0$, $\sigma_{l1}=n_l$, $\sigma_i=\sum_{1\leq j_1<\cdots<j_i\leq s} n_{j_1}\ldots n_{j_i}$, and  \\
$\sigma_{li}=\displaystyle\sum_{\substack{1\leq j_1<\cdots<j_{i-1}\leq s \\ j_1,\ldots,j_{i-1}\neq l}} n_ln_{j_1}\ldots n_{j_{i-1}}$ for $l,i=1,\ldots,s$.
\end{theorem}

\begin{proof}
{By Corollary \ref{cor1}, we have
$$S_G(\lambda)=(\lambda +1)^{n-k} \prod_{i=1}^{s}(\lambda+1-2n_i)^{r_i-1} \Big(\prod_{i=1}^{s}(\lambda+1-2n_i)+\sum_{i=1}^{s}r_in_i \prod_{j\neq i}(\lambda+1-2n_j)\Big),$$
where $n=\sum_{i=1}^{s}r_in_i$. As we did in the proof of Theorem \ref{coef}, one can see that the coefficient of $\lambda^{s-m}$ in $\prod_{i=1}^{s}(\lambda+1-2n_i)$ is
$$\sum_{i=0}^{m}(-1)^{i}2^i\binom{s-i}{m-i}\sigma_i,$$
where $\sigma_0=1$ and $\sigma_i=\sum_{1\leq j_1<\cdots<j_i\leq s} n_{j_1}\ldots n_{j_i}$ for $i=1,\ldots,s$.

Also, $r_1x_1\prod_{i=2}^{m}(1-2x_i)=\sum_{i=1}^{m}(-1)^{i-1}2^{i-1}r_1x_1\tilde{\sigma}_{i-1}$, where $\tilde{\sigma}_0=1$ and $\tilde{\sigma_i}=\sum_{2\leq j_1<\cdots<j_i\leq m} x_{j_1}\ldots x_{j_i}$ for $i=1,\ldots,m-1$. Hence $\binom{s-1}{m-1}\binom{m-1}{i-1}=\binom{s-1}{i-1}\binom{s-i}{m-i}$ implies that the coefficient of \underline{}$\lambda^{s-m}$ in $\sum_{i=1}^{s}r_in_i\prod_{j\neq i}(\lambda+1-2n_j)$ is
$$\sum_{i=1}^{m}(-1)^{i-1}2^{i-1}\binom{s-i}{m-i}(\sum_{l=1}^{s}r_l\sigma_{li}),$$
\noindent where $\sigma_{l1}=n_l$, and
$\sigma_{li}=\displaystyle\sum_{\substack{1\leq j_1<\cdots<j_{i-1}\leq s \\ j_1,\ldots,j_{i-1}\neq l}} n_ln_{j_1}\ldots n_{j_{i-1}}$ for $l,i=1,\ldots,s$.
By setting $\sigma_{l0}=0$, for $l=1,\ldots,s$, the result holds.
}
\end{proof}

\section{$S$-determined complete multipartite graphs}
In \cite{berman}, the authors have considered the complete $k$-partite graphs that are determined, up to switching, by their Seidel spectrum. They also proved that
if the different partition sets sizes are $n_1,\ldots,n_s$, and there are at least three parts of each size $n_i,$ $i=1,\ldots,s$, then the complete multipartite graph is $S$-determined (see \cite[Theorem 3.7]{berman}). Here, we extend this result. First we need the following theorems.

\begin{theorem} {\rm \cite[Theorem 3.3]{berman}}
\label{kpartite}
Let the Seidel spectrum of $G$ be equal to that of the complete $k$-partite graph $K_{n_1,\ldots,n_k}$. Then $G$ is a complete $k$-partite graph, up to switching.
\end{theorem}

\begin{theorem} {\rm \cite[Theorem 3.4]{berman}}
\label{nswit}
Let $n_1\geq n_2\geq \cdots \geq n_k\geq 1$ and $q_1\geq q_2\geq \cdots \geq q_k\geq 1$ be two different $k$-tuples, $k\geq 3$, such that $\sum_{i=1}^{k}n_i=\sum_{i=1}^{k}q_i$. Then $K_{n_1,\ldots,n_k}$ and $K_{q_1,\ldots,q_k}$ are not switching equivalent.
\end{theorem}

By Theorem \ref{kpartite}, since $K_{n_1,n_2}$ and $\overline{K_n}$ (where $n=n_1+n_2$) are switching equivalent (so Theorem \ref{nswit} does not hold for $k=2$), any complete bipartite graph is $S$-determined.
Throughout the rest of this paper, let $\sigma_i(x_1,\ldots,x_k)=\sum_{1\leq j_1<\cdots<j_i\leq k} x_{j_1}\ldots x_{j_i}$, for $i=1,\ldots,k$.

\begin{theorem}
\label{part}
Let two complete $k$-partite graphs $K_{n_1,\ldots,n_k}$ and $K_{q_1,\ldots,q_k}$ be Seidel cospectral. If $n_i=q_j$ for some $i,j$, $1\leq i,j\leq k$, then
there is a permutation $\pi$ such that $n_i=q_{\pi(i)}$ for $i=1,\ldots,k$.
\end{theorem}

\begin{proof}
{Without loss of generality, assume that $n_k=q_k$. Since $K_{n_1,\ldots,n_k}$ and $K_{q_1,\ldots,q_k}$ are Seidel cospectral, by Theorem \ref{coef}, we have
$n_1+\cdots+n_k=q_1+\cdots+q_k$ and $\sigma_i(n_1,\ldots,n_k)=\sigma_i(q_1,\ldots,q_k)$, for $i=3,\ldots,k$. Clearly,  $\sigma_i(n_1,\ldots,n_{k-1})=\sigma_i(q_1,\ldots,q_{k-1})$ for $i=1,k-1$. Assume that $\sigma_i(n_1,\ldots,n_{k-1})=\sigma_i(q_1,\ldots,q_{k-1})$, where $3\leq i\leq k-1$. Now, $\sigma_i(n_1,\ldots,n_k)=\sigma_i(q_1,\ldots,q_k)$ implies that $$n_k\sigma_{i-1}(n_1,\ldots,n_{k-1})+\sigma_i(n_1,\ldots,n_{k-1})=q_k\sigma_{i-1}(q_1,\ldots,q_{k-1})+\sigma_i(q_1,\ldots,q_{k-1}),$$ and hence $\sigma_{i-1}(n_1,\ldots,n_{k-1})=\sigma_{i-1}(q_1,\ldots,q_{k-1})$. Thus we conclude that $\sigma_i(n_1,\ldots,n_{k-1})=\sigma_i(q_1,\ldots,q_{k-1})$, for
$i=1,\ldots,k-1$. Therefore two polynomials $\prod_{i=1}^{k-1}(x-n_i)$ and $\prod_{i=1}^{k-1}(x-q_i)$ have the same coefficients and so there is a permutation $\pi$ such that $n_i=q_{\pi(i)}$ for $i=1,\ldots,k-1$.}
\end{proof}

The next theorem is an extension of \cite[Theorem 3.7]{berman}.

\begin{theorem}
\label{mult}
The complete $k$-partite graph $K_{n_1,\ldots,n_k}=K_{r_1.n_1,\ldots,r_s.n_s}$, where $r_i\geq 3$ for some $i=1,\ldots,s$, is $S$-determined.
\end{theorem}

\begin{proof}
{Suppose that $n_1\geq n_2\geq \cdots \geq n_s$ and $r_l\geq 3$, where $l=1,\ldots,s$. Let $r_0=0$ and $t_l=\sum_{j=0}^{l-1}r_j$. If $\lambda_1\geq \cdots \geq \lambda_{k-1}$ are the positive Seidel eigenvalues of the complete $k$-partite graph $K_{n_1,\ldots,n_k}=K_{r_1.n_1,\ldots,r_s.n_s}$, then by \cite[Lemma 3.2]{berman} we get that
$$2n_l-1\leq \lambda_{t_l+2}\leq 2n_l-1\leq \lambda_{t_l+1}\leq 2n_l-1.$$
Thus $\lambda_{t_l+1}=\lambda_{t_l+2}=2n_l-1$. By Theorem \ref{kpartite}, if a graph $G$ has the same Seidel spectrum, then $G$ is switching equivalent to a complete $k$-partite graph $K_{q_1,\ldots,q_k}$. Assume that $q_1\geq q_2\geq \cdots \geq q_k$. Again by \cite[Lemma 3.2]{berman},
$$\lambda_{t_l+2}\leq 2q_{t_l+2}-1\leq \lambda_{t_l+1}.$$
Hence $2n_l-1=2q_{t_l+2}-1$, and so $n_l=q_{t_l+2}$. Now, by Theorem \ref{part}, the result holds.
}
\end{proof}

\begin{ex}{\rm
Let $n_1\geq n_2\geq \cdots \geq n_{k-2}\geq 1$. The graph $K_{n_1,\ldots,n_{k-2},1,1}$ is $S$-determined. To see this assume that $\lambda_1\geq \cdots \geq \lambda_{k-1}$ are the positive Seidel eigenvalues of $K_{n_1,\ldots,n_{k-2},1,1}$. By \cite[Lemma 3.2]{berman}, $\lambda_{k-1}=1$. If for $q_1\geq q_2\geq \cdots \geq q_k$ the graph $K_{q_1,\ldots,q_k}$ has the same Seidel spectrum, then $2q_k-1\leq 1$. Hence by Theorem \ref{part}, we are done. Similarly, one can deduce that $K_{n_1,\ldots,n_{k-3},2,2,1}$ is $S$-determined.
}
\end{ex}

In \cite{berman}, it is shown that if $\max \{n_1,n_2,n_3\}$ is prime, then the complete tripartite graph $K_{n_1,n_2,n_3}$ is $S$-determined. Let us close this paper by the following result for the complete $4$-partite graphs.

\begin{corollary}\label{cor1}
If $\max \{n_1,n_2,n_3,n_4\}$ is prime, then the complete $4$-partite graph $K_{n_1,n_2,n_3,n_4}$ is $S$-determined.
\end{corollary}

\begin{proof}
{Let $n_1=p\geq n_2\geq n_3 \geq n_4$, where $p$ is a prime number. By Theorem \ref{kpartite}, if a graph $G$ is Seidel cospectral with $K_{n_1,n_2,n_3,n_4}$, then $G$ is switching equivalent to a complete $4$-partite graph $K_{q_1,q_2,q_3,q_4}$. Thus by Theorem \ref{coef},
$\sigma_i(n_1,\ldots,n_4)=\sigma_i(q_1,\ldots,q_4)$, for $i=1,3,4$. Clearly, $\sigma_4(n_1,\ldots,n_4)=\sigma_4(q_1,\ldots,q_4)$ implies that $p|q_1q_2q_3q_4$. Without loss of generality, assume that $q_1=lp$ for some positive integer $l$. Hence we have $n_2n_3n_4=lq_2q_3q_4$. Since $\sigma_1(n_1,\ldots,n_4)=\sigma_1(q_1,\ldots,q_4)$, one can conclude that $4p\geq lp+3$ and so $l\leq 3$. If $l=1$, then $q_1=n_1$ and by Theorem \ref{part}, the result holds. Otherwise, $\sigma_3(n_1,\ldots,n_4)=\sigma_3(q_1,\ldots,q_4)$ yields that $p|(l-1)q_2q_3q_4$ and hence $p|n_2n_3n_4$. (Note that if $l=3$, then $p\geq 3$.) If $n_3=p$ or $n_4=p$, then by Theorem \ref{mult}, we are done. Therefore suppose that $n_2=p$. Since $pn_3n_4=lq_2q_3q_4$, with no loss of generality, we have two cases: $q_1=l'p^2$ or $q_1=lp, q_2=l'p$, where $l,l'$ are positive integers. If $q_1=l'p^2$, then $n_3n_4=l'q_2q_3q_4$. Now, $\sigma_3(n_1,\ldots,n_4)=\sigma_3(q_1,\ldots,q_4)$ implies that $p|q_2q_3q_4$. So $p|n_3n_4$, and by Theorem \ref{mult}, the result holds. Finally, assume that $q_1=lp$ and $q_2=l'p$. Thus $n_3n_4=ll'q_3q_4$. Since $2p+n_3+n_4=(l+l')p+q_3+q_4$, one can deduce that $l+l'\leq 3$ which implies that $l=1$ or $l'=1$. So $q_1=p$ or $q_2=p$, and by Theorem \ref{part}, the proof is complete.
}
\end{proof}

\noindent{\bf Acknowledgment.} The authors wish to thank Professor Saieed Akbari for his fruitful comments.


\end{document}